\documentclass[10pt]{article}
\usepackage{amsfonts,amssymb,amstext}
\usepackage{amsmath}
\usepackage{amstext}
\usepackage{amsthm}
\usepackage{fullpage}
\makeatletter

\newcommand{\n}{\noindent}
\newcommand{\q}{\quad}

\newcommand{\E}{\mathrm {I\!E}}

\newcommand{\al}{\alpha}

\newcommand{\f}{\infty}

\newcommand{\si}{\sigma}

\newcommand{\be}{\beta}

\numberwithin{equation}{section}

\numberwithin{equation}{section}

\newcommand{\beq}{\begin{equation}}
\newcommand{\eeq}{\end{equation}}
\newcommand{\bea}{\begin{eqnarray}}
\newcommand{\eea}{\end{eqnarray}}
\newcommand{\beaa}{\begin{eqnarray*}}
\newcommand{\eeaa}{\end{eqnarray*}}

\usepackage{enumerate}
\newtheorem{Thm}{Theorem}[section]

\newtheorem{Lem}{Lemma}[section]

\makeatletter
\def\blfootnote{\xdef\@thefnmark{}\@footnotetext}
\makeatother

\theoremstyle{definition}

\theoremstyle{remark}
\newtheorem{Rem}{\bf{Remark}}[section]

\newcommand{\ConvD}{\overset{d}{\rightarrow}}
\newcommand{\ConvFDD}{\overset{f.d.d.}{\longrightarrow}}
\newcommand{\EqFDD}{\overset{f.d.d.}{=}}
\newcommand{\EqD}{\overset{d}{=}}
\newcommand{\Cov}{\mathrm{Cov}}
\newcommand{\Var}{\mathrm{Var}}

\begin{document}

\title{Functional Limit Theorems for Toeplitz Quadratic
Functionals of Continuous time Gaussian Stationary Processes}

\author{Shuyang Bai, Mamikon S. Ginovyan, Murad S. Taqqu\\ Boston University}

\maketitle
\date

\begin{abstract}
\noindent The paper establishes weak convergence in $C[0,1]$ of
normalized stochastic processes, generated by Toeplitz type quadratic
functionals of a continuous time Gaussian stationary process,
exhibiting long-range dependence.
Both central and non-central functional limit theorems are obtained.
\end{abstract}

\vskip3mm
\n


\section{Introduction}
\label{Int}

Let $\{X(t), \ t\in\mathbb{R}\}$ be a centered real-valued
stationary Gaussian process with spectral density $f(x)$
and covariance function $r(t)$, that is,
$r(t)=\widehat f(t)=\int_\mathbb{R} e^{ix t}\,f(x)\,dx,
~ t\in\mathbb{R}.$
\n We are interested in describing the limit (as $T\to\f$) of the following
process, generated by Toeplitz type quadratic functionals of the process $X(t)$:
\begin{equation}\label{eq:quad proc}
Q_T(t)=\int_0^{Tt}\int_0^{Tt}\widehat g(u-v)X(u)X(v)\,du\,dv,\quad t\in[0,1],
\end{equation}
where
\begin{equation}\label{eq:g hat}
\widehat g(t)=\int_\mathbb{R} e^{ix t}\,g(x)\,dx, \qquad t\in\mathbb{R},
\end{equation}
is the Fourier transform of some integrable even function $g(x)$,
$x\in\mathbb{R}$. We will refer to $g(x)$ and to its Fourier transform
$\widehat g(t)$ as a generating function and generating kernel
for the process $Q_T(t)$, respectively.

The limit of the process (\ref{eq:quad proc}) is completely
determined by the spectral density $f(x)$ (or covariance function $r(t)$)
and the generating function $g(x)$ (or generating kernel $\widehat g(t)$),
and depending on their properties, the limit can be either
Gaussian (that is, $Q_T(t)$ with an appropriate normalization obeys a
central limit theorem), or non-Gaussian.
The following two questions arise naturally:
\begin{itemize}
\item[(a)]
Under what conditions on $f(x)$ (resp. $r(t)$) and $g(x)$ (resp.
$\widehat g(t)$) will the limit be Gaussian?
\item[(b)] Describe the limit process, if it is non-Gaussian.
\end{itemize}


Similar questions were considered by Fox and Taqqu \cite{FT1},
Ginovyan and Sahakyan \cite{GS1}, and Terrin and Taqqu
\cite{terrin:taqqu:1990:noncentral} in the discrete time case.

Here we work in continuous time, and establish weak convergence
in $C[0,1]$ of the process (\ref{eq:quad proc}).
The limit processes can be Gaussian or non-Gaussian.
The limit non-Gaussian process is identical to that of in the discrete time case,
obtained in \cite{terrin:taqqu:1990:noncentral}.

But first some brief history. The question (a) goes back to the classical monograph by Grenander and
Szeg\"o \cite{GS}, where the problem was considered for discrete
time processes, as an application of the authors' theory of the
asymptotic behavior of the trace of products of truncated Toeplitz
matrices (see \cite{GS}, p. 217-219).
Later the question (a) was studied by Ibragimov \cite{I} and Rosenblatt
\cite{R2}, in connection to the statistical estimation of the spectral function
$F(x)$ and covariance function $r(t)$, respectively.
Since 1986, there has been a renewed interest in both questions (a) and (b),
related to the statistical inferences for long memory processes
(see, e.g., Avram \cite{A}, Fox and Taqqu \cite{FT1}, Ginovyan and
Sahakyan \cite{GS1}, Ginovyan et al. \cite{GST}, Giraitis et al.
\cite{giraitis:koul:surgailis:2009:large}, Giraitis and Surgailis \cite{GSu},
Giraitis and Taqqu \cite{giraitis:taqqu:2001:functional}, Terrin and Taqqu \cite{TT3},
Taniguchi  and Kakizawa \cite{TK}, and references therein).
In particular, Avram \cite{A}, Fox and Taqqu \cite{FT1},
Giraitis and Surgailis \cite{GSu}, Ginovyan and Sahakyan \cite{GS1}
have obtained sufficient conditions for the Toeplitz type quadratic forms $Q_T(1)$
to obey the central limit theorem (CLT), when the model $X(t)$ is
a discrete time process.

For continuous time processes the question (a) was studied in
Ibragimov \cite{I} (in connection to the statistical estimation
of the spectral function), Ginovyan \cite{G1,G3}, Ginovyan and Sahakyan
\cite{GS2} and Ginovyan et al. \cite{GST}, where sufficient
conditions in terms of $f(x)$ and $g(x)$ ensuring central
limit theorems for quadratic functionals $Q_T(1)$ have been obtained.

The rest of the paper is organized as follows.
In Section \ref{sec:main result} we state the main results of this paper
(Theorems \ref{Thm:CLT} - \ref{Thm:NCLT}).
In Section \ref{sec:prelim} we prove a number of preliminary lemmas that
are used in the proofs of the main results.
Section \ref{sec:proof of main} contains the proofs of the main results.

Throughout the paper the letters $C$ and $c$ with or without indices
will denote positive constants whose values can change from line to line.

\section{The Main Results}\label{sec:main result}
\label{Int}
In this section we state our main results. Throughout the paper we assume
that $f, g\in L^1(\mathbb{R})$, and with no loss of generality,
that $g\ge0$ (see \cite{GS2}, \cite{GSu}).

We first examine the case of \emph{central limit theorems}, and consider the following standard normalized version of (\ref{eq:quad proc}):
\begin{equation}\label{eq:tilde Q_T(t)}
 \widetilde Q_T(t) :=T^{-1/2}\left(Q_T(t)-\E[Q_T(t)]\right), \quad t\in[0,1].
\end{equation}


Our first result , which is an extension of Theorem 1 of \cite{GS2},
involves the convergence of finite-dimensional distributions of the process
$\widetilde Q_T(t)$ to that of a standard Brownian motion.
\begin{Thm}\label{Thm:CLT}
Assume that the spectral density $f(x)$ and the generating function $g(x)$
satisfy the following conditions:
\begin{equation}
\label{eq:CLT condition 1}
f\cdot g\in L^1(\mathbb{R})\cap L^2(\mathbb{R})
\end{equation}
and
\begin{equation}\label{eq:CLT condition 2}
\E [\widetilde{Q}_T^2(1)] \rightarrow 16\pi^3\int_\infty^\infty f^2(x) g^2(x)dx
~\text{ as }T\rightarrow\infty.
\end{equation}
Then we have the following convergence of finite-dimensional distributions
\[
 \widetilde Q_T(t) \ConvFDD \sigma B(t),
\]
where $\widetilde Q_T(t)$ is as in (\ref{eq:tilde Q_T(t)}),
$B(t)$ is a standard Brownian motion, and
\beq\label{eq:sigma}
\si^2: = 16\pi^3\int_{-\infty}^\infty f^2(x) g^2(x)dx.
\eeq
\end{Thm}

To extend the convergence of finite-dimensional distributions in Theorem \ref{Thm:CLT}
to the weak convergence in the space $C[0,1]$,
we impose an additional condition on the underlying Gaussian process $X(t)$
and on the generating function $g$.
It is convenient to impose this condition in the time domain, that is,
on the covariance function $r:=\hat{f}$ and the generating kernel $a:=\hat{g}$.
The following condition is an analog of the assumption in Theorem 2.3 of  \cite{giraitis:taqqu:2001:functional}:
\begin{equation}\label{eq:CLT tightness cond}
r(\cdot) \in L^p(\mathbb{R}), \quad a(\cdot) \in L^q(\mathbb{R})
\quad\text{for some} \quad p,q\ge 1, ~ \frac{1}{p}+\frac{1}{q}\ge \frac{3}{2}.
\end{equation}
\begin{Rem}
In fact under (\ref{eq:CLT condition 1}), the condition (\ref{eq:CLT tightness cond})  is  sufficient for the convergence in (\ref{eq:CLT condition 2}). Indeed, let $\bar{p}=p/(p-1)$ be the  H\"{o}lder conjugate of $p$ and let $\bar{q}=q/(q-1)$ be the H\"{o}lder conjugate of $q$. Since $1\le p,q\le 2$, one has by the Hausdorff-Young inequality  and (\ref{eq:CLT tightness cond}) that $
\|f\|_{\bar{p}}\le c_p\|r\|_p,~ \|g\|_{\bar{q}}\le c_q \|a\|_q, $
and hence
\[
f(\cdot) \in L^{\bar{p}},\quad g(\cdot) \in L^{\bar{q}}, \quad \frac{1}{\bar{p}}+\frac{1}{\bar{q}}=2-\frac{1}{p}-\frac{1}{q}\le 1/2.
\]
Then the convergence in (\ref{eq:CLT condition 2}) follows from  the proof of
Theorem 3 from \cite{GS2}. Note that a similar assertion in the discrete time case
was established in \cite{GSu}.
\end{Rem}
\begin{Rem}
{\rm Observe that condition (\ref{eq:CLT tightness cond}) is fulfilled if
the functions $r(t)$ and $a(t)$ satisfy the following:
there exist constants $C>0$, $\alpha^*$ and $\beta^*$,  such that
\begin{equation}\label{eq:CLT tightness cond power}
|r(t)|\le C  (1\wedge|t|^{\alpha^*-1}), \qquad |a(t)|\le C (1\wedge |t|^{\beta^*-1}),
\end{equation}
where $0<\alpha^*,\beta^*<1/2$ and $\alpha^*+\beta^*< 1/2$.
Indeed, to see this, note first that $r(\cdot), \ a(\cdot)\in L^\infty(\mathbb{R})$. Then one can choose $p,q\ge 1$ such that $p(\alpha^*-1)<-1$ and $q(\beta^*-1)<-1$, which entails that $r(\cdot)\in L^{p}(\mathbb{R})$ and $a(\cdot)\in L^{q}(\mathbb{R})$. Since $1/p+1/q<2-\alpha^*-\beta^*$ and  $2-\alpha^*-\beta^*>3/2$, one can further choose $p,q$ to satisfy $1/p+1/q \ge 3/2$.}
\end{Rem}

The next results, two functional central limit theorems, extend Theorems 1 and 5 of \cite{GS2}
to weak convergence in the space $C[0,1]$ of the stochastic
process $\widetilde Q_T(t)$ to a  standard Brownian motion.

\begin{Thm}\label{Thm:CLT C[0,1]}
Let the spectral density $f(x)$ and the generating function $g(x)$ satisfy
condition (\ref{eq:CLT condition 1}). Let the covariance function $r(t)$ and
the generating kernel $a(t)$ satisfy condition (\ref{eq:CLT tightness cond}). Then we have the following weak convergence in $C[0,1]$:
\[
 \widetilde Q_T(t) \Rightarrow \sigma B(t),
\]
where $\widetilde Q_T(t)$ is as in (\ref{eq:tilde Q_T(t)}),
$\sigma$ is as in (\ref{eq:sigma}), and $B(t)$ is a standard Brownian motion.
\end{Thm}
\vskip2mm
Recall that a function $u(x)$, $x\in\mathbb{R}$, is called slowly varying at $0$ if it
is non-negative and for any $t>0$
\[
\lim_{x\rightarrow 0} \frac{u(xt)}{u(x)}\rightarrow 1.
\]
Let $SV_0(\mathbb{R})$ be the class of slowly varying at zero functions
$u(x)$, $x\in\mathbb{R}$, satisfying the following conditions:
for some $a>0$, $u(x)$ is bounded on $[-a,a]$, $\lim_{x\to0}u(x)=0,$ \
$u(x)=u(-x)$ and $0<u(x)<u(y)$\ for\ $0<x<y<a$.
 An example of a function belonging to $SV_0(\mathbb{R})$
 is $u(x)= \left|\ln |x|\right|^{-\gamma}$ with $\gamma>0$ and $a=1$.

\begin{Thm}\label{Thm:CLT C[0,1] B}
Assume that the functions $f$ and $g$ are integrable on $\mathbb{R}$
and bounded outside any neighborhood of the origin, and satisfy for some $a>0$
\beq
\label{c-11}
f(x)\le |x|^{-\alpha}L_1(x), \q
|g(x)|\le |x|^{-\beta}L_2(x),\q x\in [-a,a]
\eeq
for some $\alpha<1, \ \beta<1$ with $\alpha+\beta\le1/2$,
where $L_1(x)$ and $L_2(x)$ are slowly varying at zero functions
satisfying
\begin{equation}\label{eq:L1 L2 cond}
L_i\in SV_0(\mathbb{R}),
\quad x^{-(\alpha+\beta)}L_i(x)\in L^2[-a,a], \quad i=1,2.
\end{equation}
 Let, in addition, the covariance function $r(t)$ and
the generating kernel $a(t)$ satisfy condition (\ref{eq:CLT tightness cond}). Then we have the following weak convergence in $C[0,1]$:
\[
 \widetilde Q_T(t) \Rightarrow \sigma B(t),
\]
where $\widetilde Q_T(t)$ is as in (\ref{eq:tilde Q_T(t)}),
$\sigma$ is as  in (\ref{eq:sigma}), and $B(t)$ is a standard Brownian motion.
\end{Thm}

\begin{Rem}\label{Rem:B follow A}
{\rm   The conditions $\al<1$ and $\be<1$ ensure that the Fourier transforms
of $f$ and $g$ are well defined. Observe that when $\al>0$ the process $\{X(t), t\in\mathbb{Z}\}$ may exhibit long-range dependence. We also allow here  $\alpha+\beta$ to assume the critical value 1/2.
}
\end{Rem}
\begin{Rem}
The assumptions $f\cdot g\in L^1(\mathbb{R})$, $f,g\in L^\infty(\mathbb{R}\setminus [-a,a])$ and (\ref{eq:L1 L2 cond}) imply that $f\cdot g \in L^2(\mathbb{R})$, so that $\sigma^2$ in (\ref{eq:sigma}) is finite.
\end{Rem}

\begin{Rem}
One may wonder, why, in Theorem \ref{Thm:CLT C[0,1] B}, we suppose that $L_1(x)$ and $L_2(x)$  belong to $SV_0(\mathbb{R})$ instead of merely being slowly varying at zero. This is done in order to deal with the critical case $\alpha+\beta=1/2$.
Suppose that we are away from this critical case, namely, $f(x)=|x|^{-\alpha}l_1(x)$ and $g(x)=|x|^{-\beta}l_2(x)$, where $\alpha+\beta< 1/2$, and $l_1(x)$ and $l_2(x)$ are slowly varying at zero functions. Assume also that $f(x)$ and $g(x)$ are integrable and bounded on $(-\infty,-a)\cup(a,+\infty)$ for any $a>0$. We claim that Theorem \ref{Thm:CLT C[0,1] B} applies. Indeed, choose $\alpha'>\alpha$, $\beta'>\beta$ with $\alpha'+\beta'<1/2$.
Write
$f(x)=|x|^{-\alpha'} |x|^{\delta}l_1(x)$, where $\delta=\alpha'-\alpha>0$. Since $l_1(x)$ is slowly varying, when $|x|$ is small enough,  for some $\epsilon\in (0,\delta)$ we have $
|x|^{\delta}l_1(x)\le |x|^{\delta-\epsilon}$. Then  one can  bound $|x|^{\delta-\epsilon}$ by $c\left|\ln |x|\right|^{-1}\in SV_0(\mathbb{R})$ for small $|x|<1$. Hence one  has  when  $|x|<1$ is small enough,
$
f(x)\le |x|^{-\alpha'} \left(c  \left|\ln |x|\right|^{-1} \right).
$
Similarly, when $|x|<1$ is small enough, one has
$
g(x)\le |x|^{-\beta'} \left(c  \left|\ln |x|\right|^{-1} \right).
$
All the assumptions in Theorem \ref{Thm:CLT C[0,1] B} are now readily
checked with $\alpha$, $\beta$ replaced by $\alpha'$ and $\beta'$, respectively.

\end{Rem}

Now we state a \emph{non-central limit theorem}  in the continuous time case.
Let the spectral density $f$ and the generating function $g$ satisfy
\beq
\label{nc-2}
f(x)= |x|^{-\al}L_1(x) \q {\rm and} \q
g(x)= |x|^{-\be}L_2(x), \q x\in \mathbb{R},\q \al<1, \, \be<1,
\eeq
 with   slowly varying at zero  functions $L_1(x)$ and $L_2(x)$
 such that $\int_{\mathbb{R}}|x|^{-\alpha}L_1(x)dx<\infty$
and $\int_{\mathbb{R}}|x|^{-\beta}L_2(x)dx<\infty$.
We assume in addition that the functions  $L_1(x)$ and $L_2(x)$ satisfy
the following condition, called Potter's bound (see \cite{giraitis:koul:surgailis:2009:large}, formula (2.3.5)):
for any $\epsilon>0$ there exists a constant $C>0$ so that if $T$ is large enough, then
\begin{equation}\label{eq:potter}
\frac{L_i(u/T)}{L_i(1/T)}\le C (|u|^{\epsilon}+|u|^{-\epsilon}), \quad i=1,2.
\end{equation}
Note that a  sufficient condition for (\ref{eq:potter}) to hold is that
$L_1(x)$ and $L_2(x)$ are bounded on intervals $[a,\infty)$ for any $a>0$,
which is the case for the slowly varying functions in Theorem \ref{Thm:CLT C[0,1] B}.

Now we are interested in the limit process of the following normalized version
of the process $Q_T(t)$ given by (\ref{eq:quad proc}), with $f$ and $g$ as in (\ref{nc-2}):
\begin{equation}\label{eq:Z_T}
Z_T(t):= \frac{1}{T^{\alpha+\beta} L_1(1/T)L_2(1/T)}\left( Q_T(t) -\E [Q_T(t)] \right).
\end{equation}

\begin{Thm}\label{Thm:NCLT}
Let $f$ and $g$ be as in (\ref{nc-2}) with $\alpha<1$, $\beta<1$
and slowly varying at zero functions $L_1(x)$ and $L_2(x)$ satisfying (\ref{eq:potter}), and let $Z_T(t)$ be as in (\ref{eq:Z_T}).
Then for $\alpha+\beta>1/2$, we have the following weak convergence
in the space $C[0,1]$:
\begin{equation*}
Z_T(t)\Rightarrow Z(t),
\end{equation*}
where the limit process $Z(t)$ is given by
\begin{equation}\label{eq:limit proc}
Z(t)=\int_{\mathbb{R}^2}'' H_t(x_1,x_2) W(dx_1) W(dx_2),
\end{equation}
with
\begin{equation}\label{eq:H_t}
H_t(x_1,x_2)=|x_1x_2|^{-\alpha/2}\int_{\mathbb{R}} \left[\frac{e^{it(x_1+u)}-1}{i(x_1+u)}\right]\cdot\left[\frac{e^{it(x_2-u)}-1}{i(x_2-u)}\right] |u|^{-\beta}du~,
\end{equation}
where $W(\cdot)$ is a complex Gaussian random measure with Lebesgue control measure, and the double prime in the integral (\ref{eq:limit proc}) indicates that the
integration excludes the diagonals $x_1=\pm x_2$.
\end{Thm}

\begin{Rem}
Comparing Theorem \ref{Thm:NCLT} and Theorem 1 of \cite{terrin:taqqu:1990:noncentral},
we see that the limit process $Z(t)$ is the same both for continuous and discrete time models.
\end{Rem}
\begin{Rem}
Denoting by $P_T$ and $P$ the measures generated in $C[0,1]$  by the processes $Z_T(t)$
and $Z(t)$ given by (\ref{eq:Z_T}) and (\ref{eq:limit proc}), respectively,
Theorem \ref{Thm:NCLT} can be restated as follows: under the conditions of
Theorem \ref{Thm:NCLT}, the measure $P_T$ converges weakly in $C[0,1]$ to
the measure $P$ as $T\rightarrow\infty$.
A similar assertion can be stated for Theorems \ref{Thm:CLT C[0,1]} and
\ref{Thm:CLT C[0,1] B}.
\end{Rem}

It is worth noting that although the statement of our Theorem \ref{Thm:NCLT} is similar to that of Theorem 1 of  \cite{terrin:taqqu:1990:noncentral}, the proof is different
and simpler, and does not use the hard analysis of \cite{terrin:taqqu:1990:noncentral},
although some technical results of \cite{terrin:taqqu:1990:noncentral}  are stated in lemmas and used  in the proofs.
Our approach in the CLT case (Theorems \ref{Thm:CLT} - \ref{Thm:CLT C[0,1] B}), uses the method developed in \cite{GS2},
which itself is based on an approximation of the trace of the product of truncated
Toeplitz operators.
For the non-CLT case (Theorem \ref{Thm:NCLT}), we use the integral representation
of the underlying process and  properties of Wiener-It\^o integrals.

\section{Preliminaries}\label{sec:prelim}

In this section we state a number of lemmas which will be used in the proof of the theorems.
The following result extends Lemma 9 of \cite{GS2}.
\begin{Lem}\label{Lem:Y(t) CLT}
Let $Y(t)$ be a centered stationary Gaussian process with spectral density $f_Y(x)\in L^1(\mathbb{R})\cap L^2(\mathbb{R})$. Consider the normalized process:
\begin{equation}\label{eq:L_T(t)}
L_T(t):=\frac{1}{T^{1/2}}\left( \int_0^{Tt} Y^2(u) du
- \E \left[\int_0^{Tt} Y^2(u)du \right]\right).
\end{equation}
Then we have the following convergence of finite-dimensional distributions:
\begin{equation}\label{eq:L_T FDD B(t)}
L_T(t) \ConvFDD \sigma_Y B(t),\qquad \sigma_Y^2 =4\pi \int_{-\infty}^\infty f^2_Y(x) dx,
\end{equation}
where $B(t)$ is standard Brownian motion.
\end{Lem}
\begin{Rem}\label{Rem:quadratic chaos}
Observe that the normalized processes $\widetilde{Q}_T(t)$  and $L_T(t)$, given by
(\ref{eq:tilde Q_T(t)}) and (\ref{eq:L_T(t)}), can be expressed by double
Wiener-It\^o integrals (see, e.g., the proof of Lemma \ref{Lem:rewrite} below).
In our proofs we will use the following fact about weak convergence of multiple Wiener-It\^o integrals:
given the convergence of the covariance, the multivariate  convergence to a
Gaussian vector is implied by the univariate convergence of each component
(see \cite{peccati:2005:gaussian}, Proposition 2).
\end{Rem}
\begin{proof}[Proof of Lemma \ref{Lem:Y(t) CLT}]
For a fixed $t$, the univariate convergence in distribution
\[
L_T(t)\ConvD   N(0,t\sigma^2_Y)\quad {\rm as} \quad T\rightarrow\infty
\]
follows from Lemma 9 of \cite{GS2}. To show (\ref{eq:L_T FDD B(t)}), in view of
Remark \ref{Rem:quadratic chaos} and Proposition 2 of \cite{peccati:2005:gaussian}, it remains
to show that the covariance structure of $L_T(t)$ converges to that of $\sigma_Y B(t)$.
Specifically, it suffices to show that for any $0<s<t$,
\begin{equation}\label{eq:Y(t) cov B(t)}
\E \left[\left(L_T(t)-L_T(s)\right)^2\right] \rightarrow \sigma_Y^2 \cdot(t-s) \quad \text{as} \quad T\rightarrow\infty.
\end{equation}
Indeed, using the fact that for a Gaussian vector $(G_1,G_2)$ we have $\Cov(G_1^2,G_2^2)=2[\Cov(G_1,G_2)]^2$, and letting
$r_Y(u)=\int_{\mathbb{R}} e^{ixu} f_Y(x) dx$ be the covariance function
of  $Y(t)$, we can write
\begin{align*}
\E \left[\left(L_T(t)-L_T(s)\right)^2\right]
= 2(t-s)\int_{-T(t-s)}^{T(t-s)} \left(1 -\frac{|u|}{T(t-s)} \right) r^2_Y(u)du.
\end{align*}
Since $f_Y(x)\in L^2(\mathbb{R})$, the Fourier transform $r_Y(u)\in L^2(\mathbb{R})$ as well. So by the Dominated Convergence Theorem and Parseval-Plancherel's  identity, we have as $T\rightarrow\infty$
\begin{equation}\label{eq:M1}
\E \left[\left(L_T(t)-L_T(s)\right)^2\right] \rightarrow 2(t-s)\int_{-\infty}^\infty r^2_Y(u) du=4\pi (t-s) \int_{-\infty}^\infty f^2_Y(x)dx=\sigma_Y^2(t-s).
\end{equation}

\end{proof}

We now discuss some results which allow one to reduce the general quadratic
functional in Theorem \ref{Thm:CLT} to a special quadratic functional
introduced in Lemma \ref{Lem:Y(t) CLT}.

By Theorem 16.7.2 from \cite{IL}, the underlying process $X(t)$ admits a
moving average representation:
\begin{equation}\label{eq:moving average rep}
X(t)=\int_{-\infty}^\infty \hat{a}(t-s)B(ds) \qquad\text{with }\quad \int_{-\infty}^\infty |\hat{a}(t)|^2 dt <\infty,
\end{equation}
where $B(t)$ is a standard Brownian motion, and $\hat{a}(t)$ is such that its inverse Fourier transform $a(x)$ satisfies
$f(x)=2\pi |a(x)|^2.$
Assuming the conditions (\ref{eq:CLT condition 1}) and (\ref{eq:CLT condition 2}), we set
$$ b(x)=(2\pi)^{1/2} a(x) (g(x))^{1/2},$$
and observe that the function $b(x)$ is then in $L^2(\mathbb{R})$ due to condition (\ref{eq:CLT condition 1}).
Consider the stationary process
\begin{equation}\label{eq:Y(t)}
Y(t)=\int_{-\infty}^\infty \hat{b}(t-s)B(ds)
\end{equation}
constructed using the Fourier transform $\hat{b}(t)$ of $b(x)$ and the \emph{same} Brownian motion $B(t)$ as in (\ref{eq:moving average rep}). The process $Y(t)$ has spectral density  (see  \cite{GS2}, equation (4.7))
\begin{equation}\label{eq:f_Y}
f_Y(x)=2\pi f(x)g(x).
\end{equation}
We have the following approximation result which immediately follows from  Lemma 10 of \cite{GS2}.
\begin{Lem}\label{Lem:approx}
Let $\widetilde Q_T(t)$ be as in (\ref{eq:tilde Q_T(t)}) and let $L_T(t)$ be as in (\ref{eq:L_T(t)}) with $Y(t)$ constructed as in (\ref{eq:Y(t)}). Then under the conditions  (\ref{eq:CLT condition 1}) and (\ref{eq:CLT condition 2}), for any $t>0$, we have
\[
\lim_{T\rightarrow\infty }\Var [ \widetilde{Q}_T(t) -L_T (t)]=0.
\]
\end{Lem}

The following lemma is a straightforward adaptation of Lemma 4.2 of \cite{giraitis:taqqu:1998:central}
for functions defined on $\mathbb{R}$.
\begin{Lem}\label{Lem:L^p sum}
If $p_j\ge 1$, $j=1,\ldots,k$, where $k\ge 2$ and $\sum_{j=1}^k \frac{1}{p_j}= k-1,$
then
\[\int_{\mathbb{R}^{k-1}} |f_1(x_1)\ldots f_{k-1}(x_{k-1}) f_{k}(x_1+\ldots+x_{k-1})| dx_1\ldots dx_k\le \prod_{j=1}^{k}\|f_j\|_{p_j}.
\]
\end{Lem}
The following lemma will be used to establish tightness in the space $C[0,1]$
in Theorem \ref{Thm:CLT C[0,1]}.
\begin{Lem}\label{Lem:square increment CLT}
Let the covariance function $r(t)$ and the generating kernel $a(t)$ satisfy condition (\ref{eq:CLT tightness cond}), and let  $\widetilde Q_T(t)$ be as in (\ref{eq:tilde Q_T(t)}). Then for all $0\le s\le t\le 1$ and $T>0$, there exists a constant $C>0$, such that
\begin{equation}\label{eq:increment est CLT}
\E\left[|\widetilde Q_T(t)-\widetilde Q_T(s)|^2 \right]\le C (t-s).
\end{equation}
\end{Lem}
\begin{proof}
For convenience we use the Wick product notation:
$
:X(u) X(v):= X(u) X(v) -\E \left[X(u)X(v)\right].
$ So for $0\le s\le t\le 1$, we can write
\begin{align*}
\widetilde Q_T(t)-\widetilde Q_T(s)&=\frac{1}{\sqrt{T}}\left(\int_{0}^{Tt}\int_0^{Tt} a(u-v) :X(u)X(v):  dudv-\int_{0}^{Ts}\int_0^{Ts} a(u-v) :X(u)X(v):  dudv \right)\notag\\
&=\frac{1}{\sqrt{T}}\int_{Ts}^{Tt}\int_{Ts}^{Tt} a(u-v) :X(u)X(v):  dudv  + \frac{2}{\sqrt{T}}\int_{0}^{Ts}\int_{Ts}^{Tt} a(u-v) :X(u)X(v):  dudv  \notag\\
:&=A(s,t,T)+B(s,t,T).
\end{align*}
Now we estimate  $B(s,t,T)$ (the function $A(s,t,T)$ can be estimated similarly). We have by Theorem 3.9 of \cite{janson:1997:gaussian} that
\begin{align*}
&\E \left[B^2(s,t,T)\right] = \frac{4}{T}\int_0^{Ts}du_1\int_{Ts}^{Tt}dv_1\int_0^{Ts}du_2\int_{Ts}^{Tt}dv_2 a(u_1-v_1)a(u_2-v_2) \E \left(:X(u_1)X(v_1): :X(u_2)X(v_2):\right)  \\
=&\frac{4}{T}\int_0^{Ts}du_1\int_{Ts}^{Tt}dv_1\int_0^{Ts}du_2\int_{Ts}^{Tt}dv_2a(u_1-v_1)a(u_2-v_2) \left[r(u_1-u_2)r(v_1-v_2)+r(u_1-v_2)r(v_1-u_2) \right]\\
:&=B_1(s,t,T)+B_2(s,t,T).
\end{align*}
By the change of variables $x_1=u_1-v_1$, $x_2=v_2-u_2$, $x_3=u_2-u_1$, $x_4=v_2$, and noting that $r(\cdot)$ and $a(\cdot)$ are even functions, we have
\[
B_1(s,t,T)\le \frac{4}{T} \int_{Ts}^{Tt}dx_4 \int_{\mathbb{R}^3} |a(x_1)a(x_2)r(x_3)r(x_1+x_2+x_3)| dx_1dx_2dx_3.
\]
Since  $|r(t)|\le r(0)$, we have $r(\cdot)\in L^\infty(\mathbb{R})$. We also have $r(\cdot)\in L^p(\mathbb{R})$ by condition (\ref{eq:CLT tightness cond}), where $1/p+1/q\ge 3/2$. The $L^p$-interpolation theorem states that if a function is in $L^{p_1}$ and $L^{p_2}$ with $0<p_1\le p_2\le \infty$, then it is in $L^{p'}$, $p_1\le p'\le p_2$. By the $L^p$-interpolation theorem, one can choose $p'\ge p$ such that $r(\cdot)\in L^{p'}(\mathbb{R})$ and
\[
\frac{1}{p'}+\frac{1}{p'}+\frac{1}{q}+\frac{1}{q}=3, ~\text{ that is, }~ \frac{1}{p'}+\frac{1}{q}=\frac{3}{2}.
\]
Then by Lemma \ref{Lem:L^p sum}, one has $
B_1(s,t,T)\le 4\|r\|_{p'}^2 \|a\|_{q}^2 (t-s). $
Similarly, one can establish the bound
$
B_2(s,t,T)\le C (t-s),
$
and hence
$
B(s,t,T)\le C(t-s).
$
So (\ref{eq:increment est CLT}) is proved.

\end{proof}

The lemmas that follow will be used in the proof of Theorem \ref{Thm:NCLT}.

\begin{Lem}\label{Lem:bound delta}
Define
\begin{equation}\label{eq:Delta_t}
\Delta_{t}(x)=\int_0^t e^{isx} ds=\frac{e^{itx}-1}{ix},
\end{equation}
Then for any $\delta\in (0,1)$, there exists a constant $c>0$ depending only on $\delta$, such that
\begin{equation}\label{eq:M3}
|\Delta_t(x)|\le c |t|^{\delta} f_{\delta}(x), \quad t\in [0,1], ~x\in \mathbb{R},
\end{equation}
where
\begin{equation}\label{eq:f_delta}
f_{\delta}(x)=
\begin{cases}
|x|^{\delta-1} &\text{ if }|x|>1;\\
1 &\text{ if } |x| \le 1.
\end{cases}
\end{equation}
\end{Lem}
\begin{proof}
In view of (\ref{eq:Delta_t}), we have $
|\Delta_t(x)|\le \int_0^t |e^{isx}|ds=t.
$
So under the constraint $t\in [0,1]$, we have
$
|\Delta_t(x)|\le t \le t^{\delta}.
$
On the other hand, from Lemma 2 from \cite{terrin:taqqu:1990:noncentral}, with some constant $C>0$, we have
$|e^{ix}-1|\le C |x|^{\delta}$,  $\delta\in(0,1).$ So
\[
|\Delta_t(x)|\le \frac{|e^{itx}-1|}{|x|}\le C |tx|^{\delta} |x|^{-1}=C t^{\delta} |x|^{\delta-1}.
\]
Combining this with (\ref{eq:f_delta}), we obtain (\ref{eq:M3}).
\end{proof}

We quote Lemma 1 of \cite{terrin:taqqu:1990:noncentral} in a special case,
convenient for our purposes.
\begin{Lem}\label{Lem:power counting}
Let $\gamma_i<1$, $\gamma_{i}+\gamma_{i+1}>1/2$, and let $\delta$ be such that
$$
0\le \delta<\frac{\gamma_i+\gamma_{i+1}}{2},
$$
 where $i=1,\ldots,4$ (with $\gamma_5=\gamma_1$). Then
\[
\int_{\mathbb{R}^4} f_\delta(y_1-y_2)f_\delta(y_2-y_3)f_\delta(y_3-y_4)f_\delta(y_4-y_1)  |y_1|^{-\gamma_1}|y_2|^{-\gamma_2}|y_3|^{-\gamma_3}|y_4|^{-\gamma_4}d\mathbf{y}<\infty,
\]
where $f_\delta(\cdot)$ is as in (\ref{eq:f_delta}).
\end{Lem}
Lemma \ref{Lem:power counting} can be used to establish the following result.
\begin{Lem}\label{Lem:L^2 1}
The function
\begin{equation}\label{eq:H_t*}
H_{t}^*(x_1,x_2):= |x_1|^{\alpha_1/2}|x_2|^{\alpha_2/2}\int_{\mathbb{R}}|\Delta_{t}(x_1+u)\Delta_{t}(x_2-u)| |u|^{-\beta}du~
\end{equation}
is in $L^2(\mathbb{R}^2)$ for all $(\alpha_1,\alpha_2,\beta)$ in the open region $\{(\alpha_1,\alpha_2,\beta):~\alpha_1,\alpha_2,\beta<1,~\alpha_i+\beta>1/2, ~i=1,2\}$.
\end{Lem}
\begin{proof}
It suffices  focus on the case where $t\in [0,1]$, otherwise a change of variable can reduce it to this case.
We have  by  suitable change of variables and Lemma \ref{Lem:bound delta} that
\begin{align*}
\|H_{t}^*\|_{L^2(\mathbb{R}^2)}^2&=\int_{\mathbb{R}^4} \big|\Delta_t(y_1-y_2)\Delta_t(y_2-y_3)\Delta_t(y_3-y_4)\Delta_t(y_4-y_1)\big|  |y_1|^{-\alpha_1}|y_2|^{-\beta}|y_3|^{-\alpha_2}|y_4|^{-\beta}d\mathbf{y}\\
&\le C \int_{\mathbb{R}^4} f_\delta(y_1-y_2)f_\delta(y_2-y_3)f_\delta(y_3-y_4)f_\delta(y_4-y_1)  |y_1|^{-\alpha_1}|y_2|^{-\beta}|y_3|^{-\alpha_2}|y_4|^{-\beta}d\mathbf{y}.
\end{align*}
Then apply Lemma \ref{Lem:power counting}, noting that $\delta$ can be chosen arbitrarily small.
\end{proof}

\begin{Lem}\label{Lem:L^2}
Define the function
\begin{equation}\label{eq:H*_{t,T}}
H^*_{t,T}(x_1,x_2)=A_{1,T}(x_1,x_2) |x_1x_2|^{-\alpha/2}\int_\mathbb{R} ~ |\Delta_{t}(x_1+u) \Delta_{t}(x_2-u)| |u|^{-\beta}  A_{2,T}(u) ~du,
\end{equation}
where
\begin{equation}\label{eq:A}
A_{1,T}(x_1,x_2)=\sqrt{\frac{L_1(x_1/T)}{L_1(1/T)}\frac{L_1(x_2/T)}{L_1(1/T)}},\quad A_{2,T}(u)=\frac{L_2(u/T)}{L_2(1/T)}.
\end{equation}
Then for large enough $T$, we have
$
H^*_{t,T}(x_1,x_2)\in L^2(\mathbb{R}^2).
$
\end{Lem}
\begin{proof}
By (\ref{eq:potter}) and (\ref{eq:A}), for any $\epsilon>0$ there exists $C>0$, such that for $T$  large enough,
\begin{equation}\label{eq:potter app 1}
|A_{1,T}(x_1,x_2)|\le C (|x_1|^\epsilon + |x_1|^{-\epsilon})(|x_2|^\epsilon+|x_2|^{-\epsilon})
\end{equation}
and
\begin{equation}\label{eq:potter app 2}
 |A_{2,T}(u)|\le C (|u|^\epsilon + |u|^{-\epsilon}).
\end{equation}
Hence, with some constant $C>0$,
\begin{align}\label{eq:bound L^2}
|H_{t,T}^*(x_1,x_2)|\le
 C  \int_{\mathbb{R}} ~|\Delta_{t}(x_1+u) \Delta_{t}(x_2-u)||u|^{-\beta} (|u|^\epsilon + |u|^{-\epsilon}) du~|x_1x_2|^{-\alpha/2} (|x_1|^\epsilon + |x_1|^{-\epsilon})(|x_2|^\epsilon+|x_2|^{-\epsilon}).
\end{align}
Because  by Lemma \ref{Lem:L^2 1}, the function  $H_t^*$ in (\ref{eq:H_t*})
is in $L^2(\mathbb{R}^2)$ for all $(\alpha_1,\alpha_2,\beta)$ in an open region $\{(\alpha,\beta): \alpha_1,\alpha_2,\beta<1, \alpha_i+\beta>1/2, i=1,2\}$.
By choosing $\epsilon$ small enough, we  infer that the right-hand side of (\ref{eq:bound L^2}) is in $L^2(\mathbb{R}^2)$, and the result follows.
\end{proof}

\begin{Lem}\label{Lem:rewrite}
Let $Z_T(t)$ be as in (\ref{eq:Z_T}), and let
\begin{align}\label{eq:rewrite}
Z_T'(t):=\int_{\mathbb{R}^2}'' ~
H_{t,T}(x_1,x_2)~  W(dx_1) W(dx_2),
\end{align}
where
\begin{equation}\label{eq:H_t,T}
H_{t,T}(x_1,x_2)=  A_{1,T}(x_1,x_2) |x_1x_2|^{-\alpha/2}~  \left[\int_\mathbb{R} ~ \Delta_{t}(x_1+u) \Delta_{t}(x_2-u) |u|^{-\beta}  A_{2,T}(u) ~du \right].
\end{equation}
Then $
Z_T(t)\EqFDD Z_T'(t),$
that is, the processes $Z_T(t)$ and $Z_T'(t)$ have the same finite-dimensional distributions.
\end{Lem}
\begin{proof}
Using the spectral representation of $X(t)$ (see, e.g., \cite{doob:1953:stochastic}, Chapter XI, Section 8):
$
X(t)=\int_{\mathbb{R}} e^{itx} \sqrt{f(x)}  W(dx),
$
where $W(\cdot)$ is a complex Gaussian measure with Lebesgue control measure,
 and the diagram formula (see, e.g.,  \cite{major:1981:multiple}, Chapter 5), we have
\[
X(u)X(v)-\E \left[X(u)X(v)\right]=\int_{\mathbb{R}^2}'' e^{i(ux_1+vx_2)}\sqrt{f(x_1)f(x_2)} W(dx_1)W(dx_2).
\]

By a stochastic   Fubini Theorem (see \cite{pipiras:taqqu:2010:regularization},  Theorem 2.1) and Lemma \ref{Lem:L^2}, one can change the integration order to get (note that by (\ref{eq:g hat}) we have $\widehat{g}(t)=\int_{\mathbb{R}}e^{itx}g(x)dx$):
\begin{align*}
Z_T(t)=&\frac{1}{T^{\alpha+\beta} L_1(1/T)L_2(1/T)}
\int_{\mathbb{R}^2}'' \sqrt{f(x_1)f(x_2)} ~\int_0^{Tt} \int_0^{Tt} \int_{\mathbb{R}}e^{i(u-v)w}g(w)dw~ e^{i(ux_1+vx_2)} du dv~W(dx_1)W(dx_2)\\
=& \frac{1}{T^{\alpha+\beta} L_1(1/T)L_2(1/T)}
\int_{\mathbb{R}^2}'' \sqrt{f(x_1)f(x_2)} ~\int_{\mathbb{R}}~\int_0^{Tt} e^{iu(x_1+w)} du\int_0^{Tt} e^{iv(x_2-w)} dv ~ |w|^{-\beta} L(w) dw  ~ W(dx_1)W(dx_2)
\\
=&\frac{1}{T^{\alpha+\beta} L_1(1/T)L_2(1/T)} \int_{\mathbb{R}^2}'' \sqrt{f(x_1)f(x_2)} \int_{\mathbb{R}}~ \Delta_{Tt}(x_1+w) \Delta_{Tt}(x_2-w)|w|^{-\beta}L_2(w)dw ~ W(dx_1)W(dx_2).
\end{align*}
Now we use the change of variables $w\rightarrow u/T$, $x_1\rightarrow x_1/T$, $x_2\rightarrow x_2/T$, where the latter two change of variables are subject to the rule $W(dx/T)\EqD T^{-1/2}W(dx)$ (see, e.g., \cite{dobrushin:1979:gaussian}, Proposition 4.2), to obtain
\begin{align}\label{eq:M2}
\nonumber
Z_T(t)\EqFDD&\frac{1}{T^{\alpha+\beta} L_1(1/T)L_2(1/T)}\times
\\&  \int_{\mathbb{R}^2}'' \sqrt{f(x_1/T)f(x_2/T)} \int_{\mathbb{R}}~ \Delta_{t}(x_1+u) \Delta_{t}(x_2-u)|w/T|^{-\beta}L_2(w/T)
T dw ~ T^{-1}W(dx_1)W(dx_2).
\end{align}
Taking into account the equality $f(x/T)=|x/T|^{-\alpha}L_1(x/T)$ and equations in (\ref{eq:A}), we see that the right hand side of (\ref{eq:M2}) coincides with (\ref{eq:rewrite}). This completes the proof.
\end{proof}

The lemmas that follow will be used to establish tightness in the space $C[0,1]$ in Theorem \ref{Thm:NCLT}.
\begin{Lem}\label{Lem:square increment}
Let $\delta$ be a fixed number within the range $(0,(\alpha+\beta)/2)$,
and let $Z_T(t)$ be as in (\ref{eq:Z_T}). Then for all $0\le s\le t\le 1$ and $T$  large enough, there exists a constant $C>0$, such that
\begin{equation}\label{eq:increment est}
\E\left[|Z_T(t)-Z_T(s)|^2\right] \le C (t-s)^{2\delta}.
\end{equation}
The same estimate also holds for the corresponding limiting process $Z(t)$ defined by (\ref{eq:limit proc}), (\ref{eq:H_t}).
\end{Lem}

\begin{proof}
First, in view of Lemma \ref{Lem:rewrite}, we have $\E\left[|Z_T(t)-Z_T(s)|^2\right]=\E\left[|Z_T'(t)-Z_T'(s)|^2\right]$.
Next, using the linearity of the multiple stochastic integral, we can write
\[
Z_T'(t)-Z_T'(s)= \int_{\mathbb{R}^2}'' H_{s,t,T}(x_1,x_2) W(dx_1)W(dx_2),
\]
where
\begin{align}\label{eq:H_{s,t,T}}
H_{s,t,T}(x_1,x_2)=
A_{1,T}(x_1,x_2)|x_1x_2|^{-\alpha/2} \int_{\mathbb{R}} \left[\Delta_t(x_1+u)\Delta_t(x_2-u)-\Delta_s(x_1+u)\Delta_s(x_2-u)\right]|u|^{-\beta} A_{2,T}(u)du.
\end{align}
The term in the brackets of the integrand in (\ref{eq:H_{s,t,T}}) can be rewritten as follows:
\begin{align*}
&\Delta_t(x_1+u)\Delta_t(x_2-u)-\Delta_s(x_1+u)\Delta_s(x_2-u)\\=&\int_0^t\int_0^t
e^{iw_1(x_1+u)} e^{iw_2(x_2-u)}dw_1dw_2-\int_0^s\int_0^s
e^{iw_1(x_1+u)} e^{iw_2(x_2-u)}dw_1dw_2       \\
=&\int_0^sdw_1\int_s^t dw_2\ldots +\int_s^tdw_1\int_0^s dw_2\ldots+\int_s^t dw_1\int_s^t dw_2\ldots\\
=&\Delta_s(x_1+u)\Delta_{t-s}(x_2-u)+\Delta_{t-s}(x_1+u)\Delta_s(x_2-u)+\Delta_{t-s}(x_1+u)\Delta_{t-s}(x_2-u).
\end{align*}

Now we apply Lemma \ref{Lem:bound delta} to get
\begin{align}
|\Delta_t(x_1+u)\Delta_t(x_2-u)-\Delta_s(x_1+u)\Delta_s(x_2-u)|&\le
C[s^{\delta}(t-s)^{\delta}+(t-s)^{\delta}s^{\delta}+(t-s)^{2\delta}] f_\delta(x_1+u)f_\delta(x_2-u)\notag\\
&\le C (t-s)^{\delta} f_\delta(x_1+u)f_\delta(x_2-u),\label{eq:bound delta diff}
\end{align}
where the last inequality follows because $0\le s^\delta\le 1$ and $0\le (t-s)^\delta\le 1$.

Next, using  formula (4.5$'$) of \cite{major:1981:multiple}, (\ref{eq:H_{s,t,T}}) and (\ref{eq:bound delta diff}), we can write
\begin{align}
&\E\left[ |Z_T(t)-Z_T(s)|^2\right]= \|H_{s,t,T}\|_{L^2(\mathbb{R}^2)}^2
\le  C|t-s|^{2\delta}\int_{\mathbb{R}^2} dx_1dx_2 A_{1,T}(x_1,x_2)^2 |x_1x_2|^{-\alpha} \times \notag
\\&\int_{\mathbb{R}^2}du_1du_2 f_\delta(x_1+u_1)f_\delta(x_2-u_1) f_\delta(-x_1+u_2)f_\delta(-x_2-u_2) |u_1|^{-\beta} |u_2|^{-\beta} A_{2,T}(u_1)A_{2,T}(u_2)\notag\\
\le & C|t-s|^{2\delta}\int_{\mathbb{R}^4} dy_1dy_2dy_3dy_4 A_{1,T}(y_1,y_3)^2 A_{2,T}(y_2)A_{2,T}(y_4)   \times \notag
\\& f_\delta(y_1-y_2) f_\delta(y_2-y_3)f_\delta(y_3-y_4) f_\delta(y_4-y_1)|y_1|^{-\alpha}|y_2|^{-\beta} |y_3|^{-\alpha} |y_4|^{-\beta}, \label{eq:increment est intermediate}
\end{align}
where we have applied the change of variables: $y_1=x_1$, $y_2=-u_1$, $y_3=-x_2$, $y_4=u_2$.

Since by assumption $\alpha<1$, $\beta<1$ and $\alpha+\beta>1/2$, and the exponent $\epsilon$  in (\ref{eq:potter app 1}) and  (\ref{eq:potter app 2}) can be chosen arbitrarily small,  for a fixed  $\delta$ satisfying $0<\delta<(\alpha+\beta)/2$, we can apply Lemma 1 of \cite{terrin:taqqu:1990:noncentral} to conclude that the integral
\[
\int_{\mathbb{R}^4} d\mathbf{y} A_{1,T}(y_1,y_3)^2 A_{2,T}(y_2)A_{2,T}(y_4)    f_\delta(y_1-y_2) f_\delta(y_2-y_3)f_\delta(y_3-y_4) f_\delta(y_4-y_1)|y_1|^{-\alpha}|y_2|^{-\beta} |y_3|^{-\alpha} |y_4|^{-\beta}
\]
is bounded for sufficiently large  $T$, which in view of (\ref{eq:increment est intermediate}) implies (\ref{eq:increment est}). The proof for $Z_T(t)$ is thus complete.
The proof for $Z(t)$ is similar and so we omit the details.
\end{proof}

\section{Proof of Main Results}\label{sec:proof of main}

\begin{proof}[Proof of Theorem \ref{Thm:CLT}]
By Lemma \ref{Lem:approx},  for any $0\le t_1<\ldots < t_n$, and constants $c_1,\ldots,c_n$, we have
\[
\lim_{T\rightarrow \infty}\Var\left[\sum_{j=1}^n  c_j \left(\widetilde Q_T(t_j) - L_T(t_j) \right)\right]=0.
\]
Therefore the convergence of finite-dimensional distributions of $\widetilde Q_T(t)$ to that of Brownian motion $\sigma B(t)$ follows from Lemma  \ref{Lem:Y(t) CLT} with $f_Y(\cdot)$ given in (\ref{eq:f_Y}) and the Cram\'er-Wold Device.
\end{proof}

\begin{proof}[Proof of Theorem \ref{Thm:CLT C[0,1]}]
In view of the well-known Prokhorov's Theorem (see, e.g., \cite{Billingsley}, p.\ 58),
to prove the theorem, we need to show convergence of finite-dimensional distributions and tightness.
The former has been established in Theorem \ref{Thm:CLT}.
To prove tightness, observe that by Lemma \ref{Lem:square increment CLT} and the
hypercontractivity inequality of the multiple Wiener-It\^o integrals
(see \cite{major:1981:multiple}, Corollary 5.6),
for any $T>0$ and  $0\le s\le t\le 1$, there exists a constant $C>0$ to satisfy
\begin{align}\label{eq:fourth moment est CLT}
\E \left[|\widetilde Q_{T}(t)-\widetilde Q_T(s)|^4\right]
\le C_2 \left(\E\left[|\widetilde Q_T(t)-\widetilde Q_T(s)|^2\right] \right)^2 \le C  (t-s)^{2}.
\end{align}
Now the tightness of the family of measures generated by the processes
$\{\widetilde Q_T(t):T>0\}$ in $C[0,1]$ follows from Lemma 5.1 of \cite{I}.
\end{proof}

\begin{proof}[Proof of Theorem \ref{Thm:CLT C[0,1] B}]
The convergence of finite-dimensional distributions follows from Theorem \ref{Thm:CLT}.
In fact, the assumptions on $f$ and $g$ in Theorem \ref{Thm:CLT C[0,1] B}
imply the conditions (\ref{eq:CLT condition 1}) and (\ref{eq:CLT condition 2})
in Theorem \ref{Thm:CLT}  (see the proof of Theorem 5 in \cite{GS2}).
The tightness can be shown similarly as in the proof of Theorem \ref{Thm:CLT C[0,1]}.
\end{proof}

\begin{proof}[Proof of Theorem \ref{Thm:NCLT}]

As in the proof of Theorem \ref{Thm:CLT C[0,1]}, we need to show convergence
of finite-dimensional distributions and tightness.
We first prove the convergence of finite-dimensional distributions, that is,
$Z_{T}(t)\ConvFDD Z(t)$ as $T\rightarrow\infty$, where $Z_{T}(t)$ and $Z(t)$
are defined by (\ref{eq:Z_T}) and (\ref{eq:limit proc}), respectively.

By Lemma \ref{Lem:rewrite}, the process $Z_T(t)$ defined in (\ref{eq:Z_T}) has the same
finite-dimensional distributions as the process $Z_T'(t)$ defined in (\ref{eq:rewrite}).
Therefore, taking into account the linearity of multiple Wiener-It\^o integral,
 and applying Cr\'amer-Wold device,
to prove $Z_{T}(t)\ConvFDD Z(t)$, it is enough to show that as $T\rightarrow\infty$,
\begin{equation}\label{eq:goal}
H_{t,T}(x_1,x_2) ~ {\rightarrow}~ H_t(x_1,x_2) \quad\text{ in }\quad L^2(\mathbb{R}^2),
\end{equation}
where $H_t(x_1,x_2)$ and $H_{t,T}(x_1,x_2)$ are as in (\ref{eq:H_t}) and (\ref{eq:H_t,T}),
respectively.

First, we show pointwise convergence for a.e.\  $(x_1,x_2)\in\mathbb{R}^2$,
that is,
\begin{align}
H_{t,T}(x_1,x_2)&= A_{1,T}(x_1,x_2) |x_1x_2|^{-\alpha/2}\int_{\mathbb{R}} \Delta_{t}(x_1+u) \Delta_{t}(x_2-u) |u|^{-\beta}A_{2,T}(u)du\label{eq:H_t,T conv} \\
& \rightarrow H_t(x_1,x_2)= |x_1x_2|^{-\alpha/2}\int_{\mathbb{R}} \Delta_{t}(x_1+u) \Delta_{t}(x_2-u) |u|^{-\beta} du\quad \text{as $T\rightarrow\infty$.}\label{eq:H_t,T limit}
\end{align}
Because $L_1(x)$ is a slowly varying function, we have $A_{1,T}(x_1,x_2)\rightarrow 1$ as $T\rightarrow\infty$, where $A_{1,T}$ is as in (\ref{eq:A}). To show that the integral in (\ref{eq:H_t,T conv}) converges to the integral in (\ref{eq:H_t,T limit}), note first that by (\ref{eq:A}), $A_{2,T}(u)\rightarrow 1$ as $T\rightarrow\infty$ because $L_2(x)$ is  a slowly varying function. Hence one only needs to bound the integrand properly and apply the Dominated Convergence Theorem.  To this end, observe that by
(\ref{eq:potter app 2}) for $T$ large enough, we have
\begin{align}\label{eq:g_T}
g_T(u;x_1,x_2):&=|\Delta_{t}(x_1+u)| |\Delta_{t}(x_2-u)| |u|^{-\beta}A_{2,T}(u)\\&\le C  |\Delta_{t}(x_1+u)| |\Delta_{t}(x_2-u)| |u|^{-\beta} (|u|^\epsilon+|u|^{-\epsilon}):=g_\epsilon(u;x_1,x_2).
\end{align}
By choosing $\epsilon$ small enough, using Fubini Theorem and Lemma \ref{Lem:L^2 1},
we conclude  that $g_\epsilon(\cdot\,;x_1,x_2)\in L^1(\mathbb{R})$ for a.e. $(x_1,x_2)\in \mathbb{R}^2$.
Now (\ref{eq:goal}) follows from (\ref{eq:bound L^2}) and the Dominated Convergence Theorem.

To prove tightness, first observe that by the
hypercontractivity inequality of the multiple Wiener-It\^o integrals
(see \cite{major:1981:multiple}, Corollary 5.6) and Lemma \ref{Lem:square increment},
for $T$ large enough and for any $0\le s\le t\le 1$, there exists a constant $C>0$ to satisfy
\begin{align}\label{eq:fourth moment est}
\E \left[|Z_{T}(t)-Z_T(s)|^4\right] \le C_2 \left(\E\left[|Z_T(t)-Z_T(s)|^2\right]\right)^2 \le C  |t-s|^{4\delta},
\end{align}
where $\delta$ is a fixed number within the range $0<4\delta<2(\alpha+\beta)$. Since by assumption $\alpha+\beta>1/2$, we can choose $\delta$ to satisfy $4\delta>1$.
Inequalities similar to (\ref{eq:fourth moment est}) hold also for the limit process $Z(t)$.

In view of (\ref{eq:fourth moment est}) and a similar inequality for $Z(t)$,
it follows from Kolmogorov's criterion (see, e.g., \cite{bass:2011:stochastic}
Theorem 8.1(1)) that the processes $Z_T(t)$ and $Z(t)$ admit continuous versions
when $T$ is large enough.

Now the tightness of the family of measures generated by the processes
$\{Z_T(t): T>0\}$ in $C[0,1]$ follows from Lemma 5.1 of \cite{I}.
Theorem \ref{Thm:NCLT} is proved.
\end{proof}

\noindent\textbf{Acknowledgement} The research was partially supported by National Science Foundation
Grant \#DMS-1309009 at Boston University. We would also like to thank the referee for his comments.

\bigskip

\noindent Shuyang Bai~~~~~~~ \textit{bsy9142@bu.edu}\\
Mamikon S.　Ginovyan ~\textit{ginovyan@math.bu.edu}\\
Murad S. Taqqu ~~\textit{murad@bu.edu}\\
Department of Mathematics and Statistics\\
111 Cumminton Street\\
Boston, MA, 02215, US


\begin{thebibliography}{999}

\bibitem{A}
Avram, F.: On bilinear forms in Gaussian random variables and
Toeplitz matrices. \textit{Probability Theory and Related Fields.} {79}, 37 - 45, 1988.

\bibitem{bass:2011:stochastic}
Bass, R.F.:
\newblock \emph{Stochastic Processes}.
\newblock Cambridge University Press, 2011.

\bibitem{Billingsley}
Billingsley, P.: \emph{Convergence of Probability Measures}. New York, NY: John Wiley \& Sons, Inc, 1999.

\bibitem{dobrushin:1979:gaussian}
Dobrushin, R.L.:
\newblock {G}aussian and their subordinated self-similar random generalized
  fields.
\newblock \emph{The Annals of Probability}, 7\penalty0 (1):\penalty0 1--28,
  1979.

\bibitem{doob:1953:stochastic}
Doob, J.L.:
\newblock \emph{Stochastic Processes}, volume 101.
\newblock New York Wiley, 1953.

\bibitem{FT1}
Fox, R., Taqqu, M.S.: Central limit theorem for quadratic forms
in random variables having long--range dependence. \textit{Probability Theory and
Related Fields.} {74}, 213 - 240, 1987.

\bibitem{G1}
Ginovian, M.S.: On estimate of the value of the linear
functional in a spectral density of stationary Gaussian
process. \textit{Theory of Probability and Its Applications} {33}, 777 - 781, 1988.

\bibitem{G3}
Ginovian, M.S.: On Toeplitz type quadratic functionals in
Gaussian stationary process. \textit{Probability Theory and
Related Fields.}
{100}, 395 - 406, 1994.

\bibitem{GS1}
Ginovyan, M.S., Sahakyan, A. A.: On the Central Limit
Theorem for Toeplitz quadratic forms of stationary sequences.
\textit{Theory of Probability and Its Applications.} {49}, 612 - 628, 2005.

\bibitem{GS2}
Ginovyan, M.S., Sahakyan, A.A.: Limit theorems for Toeplitz
quadratic functionals of continuous-time stationary process.
\textit{Probability Theory and
Related Fields.} 138, 551--579, 2007.

\bibitem{GST}
Ginovyan, M.S., Sahakyan, A.A. and  Taqqu, M.S.:
The trace problem for Toeplitz matrices and operators and its
impact in probability.  \textit{Probability Surveys.} {11}, 393-440, 2014.

\bibitem{giraitis:koul:surgailis:2009:large}
Giraitis, L, Koul, H.L. and Surgailis, D.:
\newblock \emph{Large Sample Inference for Long Memory Processes}.
\newblock World Scientific Publishing Company Incorporated, 2012.

\bibitem{GSu}
Giraitis, L., Surgailis, D.: A central limit theorem for quadratic
forms in strongly dependent linear variables and its application
to asymptotical normality of Whittle's estimate. \textit{Probability Theory and
Related Fields} {86}, 87--104, 1990.

\bibitem{giraitis:taqqu:1998:central}
 Giraitis, L. and  Taqqu, M.S.:  Central limit theorems for quadratic forms with time-domain conditions. \textit{Annals of Probability.} 26(1), 377--398, 1998.

\bibitem{giraitis:taqqu:2001:functional}
Giraitis, L. and Taqqu, M.S:  Functional non-central and central limit theorems for bivariate Appell polynomials.  \textit{Journal of Theoretical Probability}, 14(2), 393--426, 2001.

\bibitem{GS}
Grenander, U., Szeg\"o,  G.:  \textit{Toeplitz Forms and Their Applications.}
University of California Press, 1958.

\bibitem{I}
Ibragimov,  I.A.: On estimation of the spectral function of a
stationary Gaussian process. \textit{Theory of Probability and Its Applications.} { 8},
391 - 430, 1963.

\bibitem{IL}
Ibragimov,  I.A., Linnik, Yu. V.: \textit{Independent and Stationary Sequences
 of Random Variables.} Wolters-Noordhoff Publishing Groningen,
 The Netherlands, 1971.

\bibitem{janson:1997:gaussian}
Janson, S.
\newblock \emph{Gaussian Hilbert Spaces}, volume 129 of \emph{Cambridge Tracts
  in Mathematics}.
\newblock Cambridge University Press, 1997.

\bibitem{major:1981:multiple}
Major P.:
\newblock \emph{Multiple {W}iener-{I}t{\^o} Integrals}.
\newblock Lecture Notes in Mathematics. Springer, 1981.


\bibitem{peccati:2005:gaussian}
Peccati, G. and Tudor, C.:
\newblock Gaussian limits for vector-valued multiple stochastic integrals.
\newblock {\em S{\'e}minaire de Probabilit{\'e}s XXXVIII}, pages 219--245,
  2005.

\bibitem{pipiras:taqqu:2010:regularization}
Pipiras, V. and Taqqu, M.S.:
\newblock Regularization and integral representations of {H}ermite processes.
\newblock \emph{Statistics and Probability Letters}, 80\penalty0 (23):\penalty0
  2014--2023, 2010.


\bibitem{R2}
{Rosenblatt M.}:
Asymptotic behavior of eigenvalues of Toeplitz forms.
\textit{Journal of Applied Mathematics and Mechanics} {11} 941--950, 1962.



\bibitem {TK}
{Taniguchi, M.} and  {Kakizawa, Y.}
\textit{Asymptotic Theory of Statistical
Inference for Time Series.} Academic Press, New York, 2000.



\bibitem{terrin:taqqu:1990:noncentral}
Terrin, N and Taqqu, M.S.
\newblock A noncentral limit theorem for quadratic forms of {G}aussian
  stationary sequences.
\newblock \emph{Journal of Theoretical Probability}, 3\penalty0 (3):\penalty0
  449--475, 1990.

\bibitem{TT3}
{Terrin, N.} and {Taqqu, M.S.}
Convergence to a Gaussian limit as
the normalization exponent tends to 1/2.
\textit{Statistics and Probability Letters} {11} 419--427, 1991.



\end{thebibliography}
\end{document}